\documentclass[10pt]{article}

\usepackage{amsfonts}
\usepackage{amsmath,amssymb,amsthm,amsthm}
\usepackage{microtype}
\usepackage{hyperref}
\usepackage{cleveref}
\usepackage{geometry}
\usepackage{xcolor}
\usepackage{tikz}
\usepackage{pgfplots}
\usepackage{pgfplotstable}
\usepackage{subfig}

\usetikzlibrary{matrix,pgfplots.external,fit,math}
\DisableLigatures{encoding = *, family = * }

\newtheorem{theorem}{Theorem}[section]
\numberwithin{equation}{section}
\numberwithin{theorem}{section}

\newcommand{\norm}[2]{{\left\|#1\right\|}_{#2}}
\newcommand{\ccs}{c_{1,s}}
\newcommand{\ffl}[2]{(-d_x^{\,2})^{#1}#2}
\newcommand{\ue}[1]{#1^{\,\varepsilon}}

\newcommand{\RR}{\mathbb{R}}

\newcommand{\PP}{\mathcal{P}}
\newcommand{\uin}{u_{\textrm{\small in}}}

\title{\bf{WKB expansion for a fractional Schr\"odinger equation with applications to controllability}}

\author{Umberto Biccari \thanks{DeustoTech, University of Deusto, 48007 Bilbao, Basque Country, Spain.}\;\;\thanks{Facultad Ingenier\'{\i}a, Universidad de Deusto, Avda Universidades 24, 48007 Bilbao, Basque Country, Spain. Email: \texttt{umberto.biccari@deusto.es}}
	\and
	Alejandro B. Aceves \thanks{Southern Methodist University, Dedman College of Humanities and Sciences, PO Box 750235, Dallas, Texas, United States. Email: \texttt{aaceves@smu.edu}}
}
\date{}

\begin{document}
\bibliographystyle{acm}	
\maketitle

\begin{abstract}
This paper is devoted to the analysis of propagation properties for the solutions of a one-dimensional non-local Schr\"odinger equation involving the fractional Laplace operator $\ffl{s}{}$, $s\in(0,1)$. We adopt a classical WKB approach and we provide a systematic procedure for building a suitable ansatz for the solutions to the problem. In this way, we can obtain quasi-solutions which are localized along the rays of geometric optics, whose group velocity can be computed explicitly in terms of the parameter $s$.  Our results are then confirmed by numerical simulations, based on a finite element approximation of the fractional Laplacian and on a Crank-Nicholson scheme for the time integration. As an application, the controllability problem for the fractional Schr\"odinger equation is analyzed, finding confirmations of previously known results.
\end{abstract}

\section{Introduction}\label{intro}

In this paper, we are interested in the construction of ray-like solutions in geometric optics for the following one-dimensional non-local equation
\begin{equation}\label{main_eq}
	\PP_s u:= \left[i\partial_t + \ffl{s}{}\right]u = 0, \;\;\;\;  (x,t)\in\RR\times(0,+\infty), 
\end{equation}
with highly oscillatory initial datum
\begin{equation}\label{in_dat}
	u(x,0) = \uin(x) e^{i\frac{\xi_0}{\varepsilon} x}:=u_0(x),\;\;\; \xi_0\in\RR.
\end{equation}

In \eqref{main_eq}, $\ffl{s}{}$ is the fractional Laplacian, defined for all $s\in(0,1)$ and for any function $f$ sufficiently smooth as the following singular integral
\begin{align*}
	\ffl{s}{f}(x):=\ccs\; P.V. \int_{\RR}\frac{f(x)-f(y)}{|x-y|^{1+2s}}\,dy,
\end{align*}
with $\ccs$ a normalization constant given by (see \cite[Section 3]{dihitchhiker} and \cite[Appendix A]{ros2014pohozaev})
\begin{equation}\label{ccs}
	\ccs:= \left(\int_{\RR} \frac{1-\cos(z)}{|z|^{1+2s}}\,dz\right)^{-1} = \frac{s2^{2s}\Gamma\left(s+\frac 12\right)}{\sqrt{\pi}\Gamma(1-s)},
\end{equation}
where $\Gamma$ is the usual Gamma function. The parameter $\varepsilon$ in \eqref{in_dat} represents the typical wavelength of oscillations of the initial data, which in turn will affect the solution form of \eqref{main_eq}. Moreover, we will assume the initial phase $\uin$ to be an $L^2(\RR)$ function, so that we have $u_0\in L^2(\RR)$.

Space-fractional Schr\"odinger equations have been introduced by Laskin in quantum mechanics (\cite{laskin2000quantum,laskin2000fractional,laskin2002fractional}), since they provide a natural extension of the
standard local model when the Brownian trajectories in Feynman path integrals are replaced by Levy flights, which are generated by the fractional Laplacian. Applications of \eqref{main_eq} may be found in the study of a condensed-matter realization of L\'evy crystals (\cite{stickler2013potential}). More recently, the fractional Schr\"odinger equation was introduced into optics by Longhi in \cite{longhi2015fractional}, with applications to laser implementation. 

The present paper deals with the construction of asymptotic approximations in geometric optics for the solutions to \eqref{main_eq}, and with their application to the study of controllability problems. 

Asymptotic analysis for wave-like equations through geometric optics (also known as the Wentzel-Kramers-Brillouin (WKB) method or ray-tracing, \cite{brillouin1926mecanique,kramers1926wellenmechanik,spigler1997survey,wentzel1926verallgemeinerung}) is nowadays a classical tool that has been developed in several directions. An incomplete biography on the topic includes \cite{liu2010recovery,liu2013error,liu2015sobolev}. It is by now well-known that wave-type equations, in a local framework, have solutions that are localized near curves $(t,x(t))$ in space-time, also called rays. These curves are, in the interior of the domain of definition of the equation, solutions of a Hamiltonian system of ordinary differential equations which involves the coefficients of the operator. When one of these trajectories hits the boundary of the domain it is reflected according to the classical laws of optics.

With this observation in mind, asymptotic methods allow to study the behavior of several wave-type phenomena, with applications, e.g., in geophysics (\cite{vcerveny1982computation,hill2001prestack}), acoustic wave equations (\cite{liu2009recovery,tanushev2008superpositions}) or gravity waves (\cite{tanushev2007mountain}). 

To the best of our knowledge, a WKB approach has not yet been fully developed in a non-local setting. On the other hand, this is certainly an interesting issue, not only from a purely mathematical perspective, but also due to the several applications that we mentioned above. Our work represents a first step in this direction, providing a complete procedure for obtaining a WKB expansion of equation \eqref{main_eq}. In more detail, we will show that given a ray $(t,x(t))$ it is possible to construct quasi-solutions of the fractional Schr\"odinger equation \eqref{main_eq} such that the amount of their energy outside a ball of radius $\varepsilon^{\frac 14}$ centered at $x(t)$ is of the order of $\varepsilon^{\frac 14}$. 

A motivation and a natural application for our construction will then be the study of controllability properties for the one dimensional fractional Schr\"odinger equation 
\begin{align}\label{schr_control}
	\begin{cases}
		iu_t+\ffl{s}{u} = g\chi_{\omega\times(0,T)}, & (x,t)\in (-1,1)\times(0,T)
		\\
		u\equiv 0, & (x,t)\in(-1,1)^c\times(0,T)
		\\
		u(x,0)=u_0(x), & x\in(-1,1),
	\end{cases}
\end{align}
where $\omega$ is a subset of the space domain $(-1,1)$. As we already proved in \cite{biccari2014internal} by means of spectral analysis techniques, for \eqref{schr_control} we have the following control properties:
\begin{itemize}
	\item[$\bullet$] For $s> 1/2$, null controllability holds in any finite time $T>0$. In other words, given any $u_0\in L^2(-1,1)$ there exists a control function $g\in L^2(\omega\times(0,T))$ such that the solution to \eqref{schr_control} satisfies $u(x,T)=0$.
	
	\item[$\bullet$] For $s=1/2$, the same result holds if we assume the controllability time $T$ to be large enough, i.e $T\geq T_0>0$.
	
	\item[$\bullet$] For $s<1/2$, the equation \eqref{schr_control} is not null controllable.
\end{itemize}

Through the construction of localized solution that we are going to present in this work, it will be possible to give a further confirmation to the above facts. 

The approach that we are going to use for building localized solutions is quite standard. In particular, given a plane wave solution $\ue{u}$, we will look for quasi-solutions to \eqref{main_eq} with an ansatz of the type
\begin{align*}
	\ue{z}(x,t) = \ue{u}(x,t)\ue{a}(x,t),\;\;\;\ue{a}(x,t)=\sum_{j\geq 0} \varepsilon^{pj}a_j(x,t),
\end{align*}
with $p\in\RR$, and where the functions $a_j$ have to be determined. The identification of the $a_j$-s will then be carried out 
imposing
\begin{align*}
	\PP_s\ue{z} = O(\varepsilon^{\infty}), 
\end{align*} 
thus obtaining a series of PDEs in which it will be possible to clearly separate the leading order terms, with respect to $\varepsilon$, from several remainders which will vanish as $\varepsilon\to 0$. This will generate a cascade system for the functions $a_j$, which can then be determined as the solution of certain given Partial Differential Equations.

This paper is organized as follows. In Section \ref{prel_sec}, we will introduce some preliminary notion that we shall use in our analysis. In Section \ref{ansatz_sec}, we will present the construction of the ansatz for the asymptotic expansion of the solutions to our fractional Schr\"odinger equation \ref{main_eq}. In Section \ref{loc_sec}, we will show that the quasi-solution $\ue{z}$ obtained through our procedure is a good approximation of the real solutions $u$ to our original equation. Section \ref{control_sec} will be devoted to the discussion on the application of our methodology to the study of control properties. Finally, in Section \ref{numerical_sec}, we will present some numerical simulations which confirm our theoretical results.

\section{Preliminaries}\label{prel_sec}

Before presenting the construction of our ansatz, we introduce some preliminary facts that we are going to need for our further analysis. 

First of all, a classical result on the Schr\"odinger equation tells us that for all $s\in(0,1)$ the $H^s(\RR)$-norm of the solution $u$ to \eqref{main_eq} is conserved. This is an easy consequence of the skew-adjointness of the operator $i\ffl{s}{}$, which allows to readily check that
\begin{align*}
	0 = \big\langle u_t-i\ffl{s}{u},u+\ffl{s}{u}\big\rangle_{L^2(\RR)} = \frac 12\frac{d}{dt}\left(\norm{u}{L^2(\RR)}^2 + [u]_{H^s(\RR)}^2\right) = \frac 12\frac{d}{dt}\norm{u}{H^s(\RR)}^2.
\end{align*}
In particular, $\norm{u}{H^s(\RR)}$ represents an energy for our equation.

Second, we recall here the definition of \textit{null bicharacteristics}, which will have a fundamental role in our later construction.

Given a general pseudo-differential operator $\Psi$ with principal symbol $\psi=\psi(x,t,\xi,\tau)$, a null bicharacteristic is defined to be a solution of the following system of ordinary differential equations
\begin{align*}
	\begin{cases}
		\dot{x}(\sigma) = \psi_\xi(x(\sigma),t(\sigma),\xi(\sigma),\tau(\sigma))
		\\
		\dot{t}(\sigma) = \psi_\tau(x(\sigma),t(\sigma),\xi(\sigma),\tau(\sigma))
		\\
		\dot{\xi}(\sigma) = -\psi_x(x(\sigma),t(\sigma),\xi(\sigma),\tau(\sigma))
		\\
		\dot{\tau}(\sigma) = -\psi_t(x(\sigma),t(\sigma),\xi(\sigma),\tau(\sigma))
	\end{cases}
\end{align*}
with initial data $(x(0),t(0),\xi(0),\tau(0))=(x_0,t_0,\xi_0,\tau_0)\in\RR^4$ where the value of $\tau_0\in\RR$ is chosen so that $\psi(x_0,t_0,\xi_0,\tau_0)=0$. Then, the projection of a null bicharacteristic to the physical time-space, $(t,x(t))$, traces a curve in $(0,+\infty)\times\RR$ which is called a \textit{ray} of $\Psi$. 

In the case of our fractional Schr\"odinger equation, notice that $\PP_s=i\partial_t+\ffl{s}{}$ is a pseudo-differential operator with symbol $p_s(x,t,\xi,\tau) = \tau - |\xi|^{2s}$. Therefore, the bicharacteristic system is given by
\begin{align}\label{char_syst}
	\begin{cases}
		\dot{x}(\sigma) = \pm 2s|\xi(\sigma)|^{2s-1}, & x(0)=x_0
		\\
		\dot{t}(\sigma) = 1, & t(0)=t_0
		\\
		\dot{\xi}(\sigma) = 0, & \xi(0)=\xi_0
		\\
		\dot{\tau}(\sigma) = 0, & \tau(0)=|\xi_0|^{2s}.
	\end{cases}
\end{align}

Moreover, without losing generality we may assume $t_0=0$. Then, \eqref{char_syst} can be solved explicitly, and we obtain the following expressions for the bicharacteristics
\begin{align*}
	\begin{cases}
		x(\sigma) = x_0 \pm 2s|\xi_0|^{2s-1}\sigma
		\\
		t(\sigma) = \sigma 
		\\
		\xi(\sigma) = \xi_0
		\\
		\tau(\sigma) = |\xi_0|^{2s}.
	\end{cases}
\end{align*}

In particular, the rays of $\PP_s$ are given by the curves $(t,x_0\pm 2s|\xi_0|^{2s-1}t)\in(0,+\infty)\times\RR$. Notice that, as one expects since the operator has constant coefficients, these rays are straight lines.

\section{Construction of the ansatz}\label{ansatz_sec}

This section is devoted to a heuristic exposition of the key ideas leading to the construction of ray-like solutions for our equation. 
We begin by seeking approximate solutions with an ansatz of WKB type with linear phase:
\begin{align}\label{ansatz}
	\ue{z}(x,t) = c(\varepsilon)\ue{u}(x,t)\ue{a}(x,t),\;\;\;\ue{a}(x,t)=\sum_{j\geq 0} \varepsilon^{pj}a_j(x,t),
\end{align}
where $p\in\RR$ and the functions $a_j$ have to be determined. The constant $c(\varepsilon)$, instead, will be chosen asking that the function $\ue{z}$ has $H^s(\RR)$-norm of the order $\mathcal O(1)$. We start by observing that, since for any $\xi_0\in\RR$ we have 
\begin{align}\label{fl_exp}
	\ffl{s}{e^{i\xi_0\varepsilon^{-1} x}} = |\xi_0|^{2s}\varepsilon^{-2s}e^{i\xi_0\varepsilon^{-1} x},
\end{align} 
the plane wave
\begin{align*}
	\ue{u}(x,t):= e^{i\left[\xi_0\varepsilon^{-1}x\,+\,|\xi_0|^{2s}\varepsilon^{-2s}t\,\right]}
\end{align*}
satisfies $\PP_s\ue{u}=0$. Notice that 
\begin{align*}
	e^{i\left[\xi_0\varepsilon^{-1}x\,+\,|\xi_0|^{2s}\varepsilon^{-2s}t\,\right]} = e^{i\xi_0\varepsilon^{-1}\left(x\,-\,\frac{\varepsilon^{1-2s}}{2s}x(t)\,\right)}.
\end{align*}

At this point, for identifying the correct value of the parameter $p$ in \eqref{ansatz}, and for determining the functions $a_j$, we need to compute $\PP_s\ue{z}$ and gather the terms that we obtain according to their order with respect to $\varepsilon$. Moreover, in what follows we shall employ the following expression for the fractional Laplacian of the product of two functions
\begin{align}\label{frac_lapl_product}
	\ffl{s}{(fg)} = f\ffl{s}{g} + g\ffl{s}{f} - I_s(f,g),
\end{align}
where the term $I_s$ is given by (see, e.g., \cite{biccari2017local,ros2014pohozaev})
\begin{align*}
	I_s(f,g)(x):= \ccs\,P.V. \int_{\RR^N} \frac{(f(x)-f(y))(g(x)-g(y))}{|x-y|^{1+2s}}\,dy.
\end{align*}
By means of \eqref{fl_exp} and \eqref{frac_lapl_product}, we can immediately see that
\begin{align}\label{dalambertian_asympt}
	\PP_s\ue{z} = c(\varepsilon)\ue{u}\Big[ i\ue{a} + \ffl{s}{\ue{a}} - K_s(\ue{a})\Big]  
\end{align}
with
\begin{align*}
	K_s(\ue{a}) &= \displaystyle\ccs\,P.V. \int_{\RR}\frac{1-e^{i\frac{\xi_0}{\varepsilon}(y-x)}}{|x-y|^{1+2s}}\Big[\ue{a}(x,t)-\ue{a}(y,t)\Big]\,dy 
	\\[5pt]
	&= \displaystyle\ccs\frac{\xi_0^{2s}}{\varepsilon^{2s}}\,P.V. \int_{\RR}\frac{1-e^{iq}}{|q|^{1+2s}}\bigg[\ue{a}(x,t)-\ue{a}\left(x+\frac{\varepsilon}{\xi_0}q,t\right)\bigg]\,dq.
\end{align*}

Now, employing a fractional Taylor expansion (\cite{jumarie2006modified,trujillo1999riemann}), for any $\beta\in(0,1)$ and $x\leq\theta\leq x+\frac{\varepsilon}{\xi_0}q$ we can write

\begin{align}\label{frac_taylor}
	\displaystyle\ue{a}\left(x+\frac{\varepsilon}{\xi_0}q,t\right) = \ue{a}(x,t) + \frac{\mathcal{D}^{\,\beta}\ue{a}(x,t)}{\Gamma(1+\beta)}\left(\frac{\varepsilon}{\xi_0}q\right)^{\beta} + \frac{\mathcal{D}^{2\beta}\ue{a}(x,t)}{\Gamma(1+2\beta)}\left(\frac{\varepsilon}{\xi_0}q\right)^{2\beta} + \frac{\mathcal{D}^{3\beta}\ue{a}(\theta,t)}{\Gamma(1+3\beta)}\left(\frac{\varepsilon}{\xi_0}q\right)^{3\beta},
\end{align}
where with $\mathcal{D}^{\,\beta}$ we indicate the following fractional derivative of order $\beta$
\begin{align*}
	\mathcal{D}^{\,\beta} f(x):= \frac{1}{\Gamma(1-\beta)}\int_{-\infty}^x \frac{f'(y)}{(x-y)^{\beta}}\,dy, \;\;\;\beta\in(0,1).
\end{align*} 
Thanks to \eqref{frac_taylor}, we get
\begin{align*}
	K_s(\ue{a}) = & \displaystyle-\ccs\frac{\xi_0^{2s-\beta}}{\Gamma(1+\beta)}\varepsilon^{\beta-2s}\mathcal{D}^{\,\beta}\ue{a}(x,t)\,P.V. \int_{\RR}\frac{1-e^{iq}}{|q|^{1+2s}}q^{\beta}\,dq \nonumber
	\\
	&\displaystyle -\ccs\frac{\xi_0^{2s-2\beta}}{\Gamma(1+2\beta)}\varepsilon^{2\beta-2s}\mathcal{D}^{2\beta}\ue{a}(x,t)\,P.V. \int_{\RR}\frac{1-e^{iq}}{|q|^{1+2s}}q^{2\beta}\,dq \nonumber
	\\
	&\displaystyle -\ccs\frac{\xi_0^{2s-3\beta}}{\Gamma(1+3\beta)}\varepsilon^{2\beta-3s}\mathcal{D}^{3\beta}\ue{a}(\theta,t)\,P.V. \int_{\RR}\frac{1-e^{iq}}{|q|^{1+2s}}q^{3\beta}\,dq.
\end{align*}

Moreover, we observe that since $|1-e^{iq}| = 2-2\cos(q)$, for all $q\in\RR$ and $\beta<2s/3$ the integrals in the above expression are finite. In particular, we have 
\begin{align*}
	&\displaystyle\bullet\;\;\left|P.V. \int_{\RR}\frac{1-e^{iq}}{|q|^{1+2s}}q^{\beta}\,dq\,\right| \leq 4\Gamma(\beta-2s-1)\cos\left[\frac{(2s-\beta)\pi}{2}\right], 
	\\[5pt]
	&\displaystyle\bullet\;\;\left|P.V. \int_{\RR}\frac{1-e^{iq}}{|q|^{1+2s}}q^{2\beta}\,dq\,\right| \leq 4\Gamma(2\beta-2s-1)\cos\big[(s-\beta)\pi\big], 
	\\[5pt]
	&\displaystyle\bullet\;\;\left|P.V. \int_{\RR}\frac{1-e^{iq}}{|q|^{1+2s}}q^{3\beta}\,dq\,\right| \leq 4\Gamma(3\beta-2s-1)\cos\left[\frac{(2s-3\beta)\pi}{2}\right]. 
\end{align*}
Therefore, we can rewrite
\begin{align*}
	K_s(\ue{a}) = - \mathcal{C}_\beta \varepsilon^{\beta-2s}\mathcal{D}^{\,\beta}\ue{a}(x,t) - \mathcal{C}_{2\beta} \varepsilon^{2\beta-2s}\mathcal{D}^{2\beta}\ue{a}(x,t) - \mathcal{C}_{3\beta} \varepsilon^{3\beta-2s}\mathcal{D}^{3\beta}\ue{a}(\theta,t),
\end{align*}
with 
\begin{align*}
	\mathcal{C}_{\gamma}:= \ccs\frac{\xi_0^{2s-\gamma}}{\Gamma(1+\gamma)}P.V. \int_{\RR}\frac{1-e^{iq}}{|q|^{1+2s}}q^{\gamma}\,dq,
\end{align*}
and we then obtain
\begin{align*}
	\PP_s\ue{z}(x,t) = c(\varepsilon)\ue{u}\Big[ i\ue{a}_t(x,t) + \ffl{s}{\ue{a}}(x,t) &+ \mathcal{C}_\beta \varepsilon^{\beta-2s}\mathcal{D}^{\,\beta}\ue{a}(x,t) 
	\\
	&+ \mathcal{C}_{2\beta} \varepsilon^{2\beta-2s}\mathcal{D}^{2\beta}\ue{a}(x,t) + \mathcal{C}_{3\beta} \varepsilon^{3\beta-2s}\mathcal{D}^{3\beta}\ue{a}(\theta,t)\Big].  
\end{align*}

Furthermore, let us introduce the following rescaling of the time variable $t\mapsto \tau:=\varepsilon^{2s-\beta}t$. In this way, \eqref{dalambertian_asympt} finally becomes
\begin{align}\label{asympt_exp}
	\PP_s\ue{z} = c(\varepsilon)\varepsilon^{\beta-2s}\ue{u}\sum_{j\geq 0} \varepsilon^{pj}\Big[ i\partial_\tau a_j + \mathcal{C}_\beta \mathcal{D}^{\,\beta} a_j + \varepsilon^{2s-\beta}\ffl{s}{a_j} + \mathcal{C}_{2\beta} \varepsilon^{\beta}\mathcal{D}^{2\beta}a_j + \mathcal{C}_{3\beta} \varepsilon^{2\beta}\mathcal{D}^{3\beta}a_j(\theta,\tau)\Big].  
\end{align} 
For determining from \eqref{asympt_exp} the expression of the functions $a_j$, we will now impose, for any $j\geq 0$, 
\begin{align*}
	\PP_s\ue{z} = O(\varepsilon^{\infty}), 
\end{align*} 
thus obtaining a series of PDEs in which will be possible to clearly separate the leading order terms, with respect to $\varepsilon$, from several remainders which will vanish as $\varepsilon\to 0$. During this procedure, we will also identify the values of the parameters $p$ and $\beta$ that we shall employ.

Let us start firstly with $j=0$. In this case, it is trivial to identify the leading equation, and we immediately have that the function $a_0(x,\tau)$ has to satisfy 
\begin{align}\label{eq_0}
	i\partial_\tau a_0 + \mathcal{C}_\beta \mathcal{D}^{\beta} a_0 = 0.  
\end{align}
For $j=1$, instead, choosing $p=\beta$ we obtain from \eqref{asympt_exp} and \eqref{eq_0}
\begin{align*}
	\varepsilon^\beta\Big[i\partial_\tau a_1 + \mathcal{C}_\beta \mathcal{D}^{\,\beta} a_1 + \mathcal{C}_{2\beta} \mathcal{D}^{2\beta}a_0 &+ \varepsilon^{2s-2\beta}\ffl{s}{a_0} + \mathcal{C}_{3\beta} \varepsilon^{\beta}\mathcal{D}^{3\beta}a_0(\theta,\tau) 
	\\
	&+ \varepsilon^{2s-\beta}\ffl{s}{a_1} + \mathcal{C}_{2\beta} \varepsilon^{\beta}\mathcal{D}^{2\beta}a_1 + \mathcal{C}_{3\beta} \varepsilon^{2\beta}\mathcal{D}^{3\beta}a_1(\theta,\tau)  \Big],
\end{align*} 
and we thus find that the function $a_1(x,\tau)$ has to satisfy
\begin{align}\label{eq_1}
	i\partial_\tau a_1 + \mathcal{C}_\beta \mathcal{D}^{\,\beta} a_1 + \mathcal{C}_{2\beta} \mathcal{D}^{2\beta} a_0 = 0.  
\end{align}
Let us now continue with $j=2$. In this case, we have 
\begin{align*}
	\varepsilon^{2\beta}\Big[\varepsilon^{-2\beta} & \underbrace{(i\partial_\tau a_0 + \mathcal{C}_\beta \mathcal{D}^{\,\beta} a_0)}_{=0} + \varepsilon^{2s-3\beta}\ffl{s}{a_0} + \mathcal{C}_{3\beta} \mathcal{D}^{3\beta}a_0(\theta,\tau)
	\\
	&+ \varepsilon^{-\beta}\underbrace{(i\partial_\tau a_1 + \mathcal{C}_\beta \mathcal{D}^{\,\beta} a_1 + \mathcal{C}_{2\beta} \mathcal{D}^{2\beta}a_0)}_{=0} + \varepsilon^{2s-2\beta}\ffl{s}{a_1} + \mathcal{C}_{3\beta} \varepsilon^{\beta}\mathcal{D}^{3\beta}a_1(\theta,\tau) + \mathcal{C}_{2\beta} \mathcal{D}^{2\beta}a_1
	\\
	&+ i\partial_\tau a_2 + \mathcal{C}_\beta \mathcal{D}^{\,\beta} a_2 + \varepsilon^{2s-\beta}\ffl{s}{a_2} + \mathcal{C}_{2\beta} \varepsilon^{\beta}\mathcal{D}^{2\beta}a_2 + \mathcal{C}_{3\beta} \varepsilon^{2\beta}\mathcal{D}^{3\beta}a_2(\theta,\tau) \Big].  
\end{align*} 
and we thus find that the function $a_2(x,\tau)$ has to satisfy
\begin{align}\label{eq_2}
	i\partial_\tau a_2 + \mathcal{C}_\beta \mathcal{D}^{\,\beta} a_2 + \mathcal{C}_{2\beta} \mathcal{D}^{2\beta} a_1 + \mathcal{C}_{3\beta} \mathcal{D}^{3\beta} a_0(\theta,\tau) = 0.  
\end{align}

For $j=3$, choosing $\beta=s/2$ (observe that this is admissible since $s/2<2s/3$), and employing \eqref{eq_0}, \eqref{eq_1} and \eqref{eq_2}, we obtain
\begin{align*}
	\varepsilon^{\frac{3s}{2}}\Big[i\partial_\tau a_3 + \mathcal{C}_{\frac s2}\mathcal{D}^{\frac{s}{2}}a_3 & + \mathcal{C}_{s} \mathcal{D}^{s}a_2 + \mathcal{C}_{\frac{3s}{2}} \mathcal{D}^{\frac{3s}{2}}a_1(\theta,\tau) + \ffl{s}{a_0}  + \varepsilon^{\frac{s}{2}}\ffl{s}{a_1} + \varepsilon^{s}\ffl{s}{a_2}  
	\\
	&+ \mathcal{C}_{\frac{3s}{2}} \varepsilon^{\frac{s}{2}}\mathcal{D}^{\frac{3s}{2}}a_2(\theta,\tau)+ \varepsilon^{\frac{3s}{2}}\ffl{s}{a_3} + \mathcal{C}_{s}\varepsilon^{\frac{s}{2}}\mathcal{D}^{s}a_3 + \mathcal{C}_{\frac{3s}{2}} \varepsilon^{s}\mathcal{D}^{\frac{3s}{2}}a_3(\theta,\tau)\Big].  
\end{align*}
Therefore, the function $a_3(x,\tau)$ has to satisfy
\begin{align*}
	i\partial_\tau a_3 + \mathcal{C}_{\frac s2}\mathcal{D}^{\frac{s}{2}}a_3 + \mathcal{C}_{s} \mathcal{D}^{s}a_2 + \mathcal{C}_{\frac{3s}{2}} \mathcal{D}^{\frac{3s}{2}}a_1(\theta,\tau) + \ffl{s}{a_0}.
\end{align*}

Furthermore, we have identified both the parameters $p$ and $\beta$. Thus, we can iterate the procedure described above and we find the following expression for the quasi-solution $\ue{z}$
\begin{align}\label{ansatz_prel}
	\ue{z}(x,t) = c(\varepsilon)e^{i\left[\xi_0\varepsilon^{-1}x\,+\,|\xi_0|^{2s}\varepsilon^{-2s}t\right]}\sum_{j\geq 0}\varepsilon^{\frac{s}{2}j}a_j\left(x,\varepsilon^{\frac{3}{2}s}t\right),
\end{align}  
where the functions $a_j$, $j\geq 0$, are the solutions of the following cascade system
\begin{align}\label{cascade_system}
	\begin{cases}
		i\partial_\tau a_0 + \mathcal{C}_{\frac s2} \mathcal{D}^{\frac{s}{2}} a_0 = 0  
		\\
		i\partial_\tau a_1 + \mathcal{C}_{\frac s2} \mathcal{D}^{\frac{s}{2}} a_1 + \mathcal{C}_{s} \mathcal{D}^{s} a_0 = 0  
		\\
		i\partial_\tau a_2 + \mathcal{C}_{\frac s2} \mathcal{D}^{\frac{s}{2}} a_2 + \mathcal{C}_{s} \mathcal{D}^{s} a_1 + \mathcal{C}_{\frac{3s}{2}} \mathcal{D}^{\frac{3s}{2}} a_0(\theta,\tau) = 0, & \displaystyle x\leq\theta\leq x+\frac{\varepsilon}{\xi_0}q
		\\
		i\partial_\tau a_j + \mathcal{C}_{\frac s2}\mathcal{D}^{\frac{s}{2}}a_j + \mathcal{C}_{s} \mathcal{D}^{s}a_{j-1} + \mathcal{C}_{\frac{3s}{2}} \mathcal{D}^{\frac{3s}{2}}a_{j-2}(\theta,\tau) + \ffl{s}{a_{j-3}}, & j\geq 3
		\\
		&\displaystyle x\leq\theta\leq x+\frac{\varepsilon}{\xi_0}q.  
	\end{cases}
\end{align}  

Notice that the classical Borel's theorem (see, e.g., \cite[Chapter I, Theorem 1.2.6]{hormander1990analysis}) allows one to choose a $C^\infty$-smooth function $\ue{a}(x,\tau)$ which has the expansion at $\varepsilon = 0$
\begin{align*}
	\ue{a}(x,\tau)=\sum_{j\geq 0} \varepsilon^{\frac s2 j}a_j(x,\tau).
\end{align*}

This, in particular, justifies the formal computations presented above. Consequently, we conclude that the function $\ue{z}(x,t)$
constructed in \eqref{ansatz_prel} is $C^{\infty}$-smooth and it is an infinitely accurate solution of the fractional Schr\"odinger equation in the sense that $\mathcal P_s\ue{z} = O(\varepsilon^\infty)$ in $\RR\times (0,+\infty)$.

For concluding our construction, let us now compute the value of the normalization constant $c(\varepsilon)$. This is done asking that $\norm{\ue{z}}{H^s(\RR)}=\mathcal O(1)$ as $\varepsilon\to 0^+$. First of all, from the construction that we just presented we get 
\begin{align*}
	\ue{z}= c(\varepsilon)e^{i\left[\xi_0\varepsilon^{-1}x\,+\,|\xi_0|^{2s}\varepsilon^{-2s}t\,\right]}\left(a_0 + \mathcal O(\varepsilon^{\frac s2})\right).
\end{align*}

Hence, in what follows we can consider only the term for $j=0$ in our expansion. Moreover, since we are working on the whole $\RR$, we have
\begin{align*}
	\norm{\ue{z}}{H^s(\RR)} &= \left(\norm{\ue{z}}{L^2(\RR)}^2 + \norm{\ffl{\frac s2}{\ue{z}}}{L^2(\RR)}^2\right)^{\frac 12} 
	\\
	&= \left(\norm{c(\varepsilon)\ue{u}a_0}{L^2(\RR)}^2 + \norm{c(\varepsilon)\ue{u}\left(\frac{|\xi_0|^s}{\varepsilon^s}a_0 + \ffl{\frac s2}{a_0} + R\right)}{L^2(\RR)}^2\right)^{\frac 12},
\end{align*}
where, as we did before, the reminder term $R$ can be written in the form
\begin{align*}
	R=\frac{1}{\varepsilon^s} \frac{|\xi_0|^s}{\Gamma(1+s)}\left(P.V.\,\int_\RR \frac{1-e^{iq}}{|q|^{1+2s}}q^s\,dq\right) \mathcal D^sa_0 = \frac{c(s)}{\varepsilon^s}\mathcal D^sa_0.
\end{align*}
Therefore, we obtain 
\begin{align*}
	\norm{\ue{z}}{H^s(\RR)} = \left(\norm{c(\varepsilon)\ue{u}a_0}{L^2(\RR)}^2 + \norm{c(\varepsilon)\ue{u}\Big[\varepsilon^{-s}\big(|\xi_0|^s a_0 + c(s)\mathcal D^sa_0\big) + \ffl{\frac s2}{a_0} \Big]}{L^2(\RR)}\right)^{\frac 12}.
\end{align*}
Thus, choosing the normalization constant as $c(\varepsilon)=\varepsilon^s$, we immediately obtain
\begin{align*}
	\norm{\ue{z}}{H^s(\RR)} &= \left(\varepsilon^s\norm{\ue{u}a_0}{L^2(\RR)}^2 + \norm{\ue{u}\Big(|\xi_0|^s a_0 + c(s)\mathcal D^sa_0 + \varepsilon^s\ffl{\frac s2}{a_0} \Big)}{L^2(\RR)}\right)^{\frac 12}
	\\
	&= \left(\varepsilon^s\norm{a_0}{L^2(\RR)}^2 + \norm{|\xi_0|^s a_0 + c(s)\mathcal D^sa_0 + \varepsilon^s\ffl{\frac s2}{a_0} }{L^2(\RR)}\right)^{\frac 12} = \mathcal O\big(\norm{a_0}{H^s(\RR)}\big).
\end{align*}
In this way, we find the final expression for our quasi-solution, which reads as follows 
\begin{align}\label{z-eps}
	\ue{z}(x,t) = \varepsilon^s e^{i\left[\xi_0\varepsilon^{-1}x\,+\,|\xi_0|^{2s}\varepsilon^{-2s}t\right]}\sum_{j\geq 0}\varepsilon^{\frac{s}{2}j}a_j\left(x,\varepsilon^{\frac{3}{2}s}t\right).
\end{align} 

Moreover, \eqref{cascade_system} is uniquely solvable with initial conditions imposed at $t = 0$ and this, of course, allows to identify the expressions of the functions $a_j$. In more detail, it is possible to compute quasi-solutions to the initial value problem 
\begin{align}\label{schr_asympt}
	\begin{cases}
		iu_t + \ffl{s}{u}=0, \;\,\; (x,t)\in\RR\times(0,+\infty)
		\\
		u(x,0)=u_0(x),
	\end{cases}	
\end{align}
by means of the following procedure. 

\subsubsection*{\bf{Step 1. $j=0$}}

Given any $g_0\in L^2(\RR)$, we start by considering the equation
\begin{align}\label{a0_pb}
	\begin{cases}
		i\partial_\tau a_0 + \mathcal{C}_{\frac s2} \mathcal{D}^{\frac{s}{2}} a_0 = 0, \;\;\; (x,\tau)\in\RR\times(0,+\infty)  
		\\
		a_0(x,0)=g_0(x).
	\end{cases}
\end{align}  

Recall that $\tau=\varepsilon^{\frac 32 s}t$. Moreover, it is known that the solution to \eqref{a0_pb} can be computed explicitly and it is given by 
\begin{align}\label{a0}
	a_0(x,\tau) = \int_\RR \mathcal{G}(y,\tau)g_0(x-y)\,dy,
\end{align}
where with $\mathcal{G}$ we refer to the Green's function defined as the solution to 
\begin{align}\label{green}
	\begin{cases}
		i\partial_\tau \mathcal{G} + \mathcal{C}_{\frac s2} \mathcal{D}^{\frac{s}{2}} \mathcal{G} = 0, \;\;\; (x,\tau)\in\RR\times(0,+\infty)  
		\\
		\mathcal{G}(x,0)=\delta(x).
	\end{cases}
\end{align}  

Equation \eqref{green}  can be easily solved with the help of the Fourier transform. With this in mind, let us recall that we have (see, e.g., \cite[Page 59, Equation A.13]{metzler2000random})
\begin{align*}
	\mathcal{F}\left[\mathcal{D}^{\frac s2}a_0\right](k,\tau) = -|k|^{\frac s2} \mathcal{F}\left[a_0\right](k,\tau)
\end{align*}
In particular, the function $\mathcal{G}$ is given by
\begin{align*}
	\mathcal{G}(x,\tau) = \frac{1}{2\pi}\int_\RR e^{ik x}e^{-i\mathcal{C}_{\frac s2}|k|^{\frac{s}{2}}\tau}\,dk.
\end{align*}
so one obtains from \eqref{a0} the following expression for the solution to \eqref{a0_pb}
\begin{align*}
	a_0(x,\tau) = \frac{1}{2\pi}\int_\RR \left(\int_\RR e^{ik x}e^{-i\mathcal{C}_{\frac s2}|k|^{\frac{s}{2}}\tau}\,dk\right)g_0(x-y)\,dy.
\end{align*}

\subsubsection*{\bf{Step 2. $j=1$}}

Once the expression of $a_0$ is determined, the next component in the expansion, corresponding to $j=1$ in \eqref{z-eps}, is obtained solving the non-homogeneous equation  
\begin{align}\label{a1_pb}
	\begin{cases}
		i\partial_\tau a_1 + \mathcal{C}_{\frac s2} \mathcal{D}^{\frac{s}{2}} a_1 = h, \;\;\; (x,\tau)\in\RR\times(0,+\infty)  
		\\
		a_1(x,0)=g_1(x),
	\end{cases} 
\end{align} 
where we indicated with $h$ the function
\begin{align*}
	h(x,\tau) = -\mathcal{C}_s\mathcal{D}^sa_0(x,\tau).
\end{align*}

Moreover, also the solution of \eqref{a1_pb} can be obtained explicitly employing classical splitting techniques and writing $a_1(x,t)=a_{1,1}(x,t)+a_{1,2}(x,t)$ with 
\begin{align}\label{a11_pb}
	\begin{cases}
		i\partial_\tau a_{1,1} + \mathcal{C}_{\frac s2} \mathcal{D}^{\frac{s}{2}} a_{1,1} = h, \;\;\; (x,\tau)\in\RR\times(0,+\infty)  
		\\
		a_{1,1}(x,0)=0
	\end{cases}
\end{align}
and
\begin{align}\label{a12_pb}
	\begin{cases}
		i\partial_\tau a_{1,2} + \mathcal{C}_{\frac s2} \mathcal{D}^{\frac{s}{2}} a_{1,2} = 0, \;\;\; (x,\tau)\in\RR\times(0,+\infty)  
		\\
		a_{1,2}(x,0)=g_1(x).
	\end{cases}
\end{align}

Notice that the solution to \eqref{a11_pb} can be obtained through the variation of constants formula, while the solution to \eqref{a12_pb} is computed as in Step 1.

\subsubsection*{\bf{Step 3. $j\geq 2$}}

Starting from $j=2$, in the equations determining $a_j$ appears terms in the variable $\theta$, coming from the remainders of the fractional Taylor expansion. Notice, however, that the support of $\theta$ is the interval $[x,x+\varepsilon q/\xi_0]$ and that we are interested in analyzing the behavior of the quasi-solutions as $\varepsilon\to 0^+$. Hence, without introducing significative errors, we can assume $\theta=x$. Therefore, we obtain that each $a_j$, $j\geq 2$, is the solution to the non-homogeneous problem
\begin{align}\label{aj_pb}
	\begin{cases}
		i\partial_\tau a_j + \mathcal{C}_{\frac s2} \mathcal{D}^{\frac{s}{2}} a_j = H_j, \;\;\; (x,\tau)\in\RR\times(0,+\infty)  
		\\
		a_j(x,0)=g_j(x),
	\end{cases}
\end{align} 
where the right hand sides $H_j$ are determined in terms of the functions $a_i$, $i=0,\ldots,j-1$. In particular, \eqref{aj_pb} can be solved again as we did for $a_1$ in Step 2, and we thus obtain explicit expressions for all the functions $a_j$, $j\geq 0$.

\section{Localization of the quasi-solutions along rays}\label{loc_sec}

This section is devoted to showing that the quasi-solutions that can be computed by using the ansatz that we obtained are in fact localized along rays. 

\begin{theorem}\label{loc_thm}
Let $\uin\in L^2(\RR)$ and let $\ue{z}$ be constructed employing the expansion \eqref{z-eps}, with initial data $g_j\in L^2(\RR)$. Then, for any $\varepsilon>0$ we have:
\begin{enumerate}
	\item The functions $\ue{z}$ are approximate solutions to \eqref{main_eq}:
	\begin{align}\label{in_dat_asympt}
		\norm{u_0(x)-\ue{z}(x,0)}{L^2(\RR)} = \mathcal{O}(\varepsilon^{\frac{1}{2}}),
	\end{align}
	\begin{align}\label{sol_asympt}
		\norm{u(x,t)-\ue{z}(x,t)}{L^2(\RR)} = \mathcal{O}(\varepsilon^{\frac 12}).
	\end{align}
	\item The initial energy of $\ue{z}$ remains bounded as $\varepsilon\to 0$, i.e.
	\begin{align}\label{in_energy_est}
		\norm{\ue{z}(x,0)}{H^s(\RR)}^2 \approx 1.
	\end{align}
	\item The energy of $\ue{z}$ is exponentially small off the ray $(t,x(t))$:
	\begin{align}\label{energy_ray}
		\int_{|x-x(t)|>\varepsilon^{\frac 14}} \left|\ffl{\frac s2}{\ue{z}}(x,t)\right|^2\,dx = \mathcal O(\varepsilon^{\frac 14}).
	\end{align}	
\end{enumerate}
\end{theorem} 

\begin{proof}
\textbf{Step 1: Approximation of the real solution.} First of all, from the definition \eqref{z-eps} of $\ue{z}$ we have 
\begin{align*}
	\ue{z}(x,0) = e^{i\frac{\xi_0}{\varepsilon} x}\sum_{j\geq 0} \varepsilon^{\frac s2 j}a_j(x,0) = e^{i\frac{\xi_0}{\varepsilon} x}\sum_{j\geq 0} \varepsilon^{\frac s2 j}g_j(x).
\end{align*}	
By means of the above expression we obtain
\begin{align*}
	\norm{u_0-\ue{z}(x,0)}{L^2(\RR)} &= \norm{u_0-e^{i\frac{\xi_0}{\varepsilon} x}\left(g_0 + \mathcal O(\varepsilon^{\frac s2})\right)}{L^2(\RR)} = \norm{e^{i\frac{\xi_0}{\varepsilon} x}\Big[\uin-\left(g_0+\mathcal O(\varepsilon^{\frac s2})\right)\Big]}{L^2(\RR)}
	\\
	&= \mathcal O(\varepsilon^{\frac 12})\norm{e^{ix}\Big[\uin-\left(g_0+\mathcal O(\varepsilon^{\frac s2})\right)\Big]}{L^2(\RR)} = \mathcal O(\varepsilon^{\frac 12})\left(\norm{\uin-g_0}{L^2(\RR)}+\mathcal O(\varepsilon^{\frac s2})\right).
\end{align*}
	
Therefore, since both $\uin$ and $g_0$ belong to $L^2(\RR)$, we immediately have \eqref{in_dat_asympt}. In order to prove \eqref{sol_asympt}, let us firstly remark that, by means of classical PDE techniques, we can obtain the following energy estimate for the solution to \eqref{main_eq} (see, e.g., \cite{cazenave2003semilinear})
\begin{align}\label{energy_est}
	\norm{u(t)}{L^2(\RR)}\leq C\left(\norm{u_0}{L^2(\RR)} + \norm{\PP_s u(t)}{L^2(\RR)}\right). 
\end{align}
Moreover, notice that since $\PP_s$ is linear, by construction of $\ue{z}$ we have 
\begin{align*}
	\PP_s(u-\ue{z}) = \PP_s u - \PP_s\ue{z} = - \PP_s\ue{z} = \mathcal O(\varepsilon^\infty).
\end{align*}
Therefore, applying \eqref{energy_est} we immediately get
\begin{align*}
	\norm{u(x,t)-\ue{z}(x,t)}{L^2(\RR)} \leq C\left(\norm{u_0-\ue{z}(x,0)}{L^2(\RR)} + \norm{\PP_s (u-\ue{z})(t)}{L^2(\RR)}\right) = \mathcal O(\varepsilon^{\frac 12}) + \mathcal O(\varepsilon^\infty),
\end{align*}
and this clearly yields \eqref{sol_asympt}. 
	
\textbf{Step 2: Initial energy estimate.} Repeating the computations developed in Section \ref{ansatz_sec} for deriving the appropriate value of the normalization constant $c(\varepsilon)$, we can easily check that 
\begin{align*}
	\norm{\ue{z}(x,0)}{H^s(\RR)}^2 \approx \norm{g_0}{H^s(\RR)}^2.
\end{align*}
Hence, up to a rescaling in the initial datum $g_0\mapsto g_0/\norm{g_0}{H^s(\RR)}$, we have \eqref{in_energy_est}.
	
\textbf{Step 3: Localization along rays.} Following the same approach that we used in precedence, and employing the change of variables $\varepsilon^{\frac 14}(x-x(t))\mapsto z$, we have
\begin{align*}
	\int_{|x-x(t)|>\varepsilon^{\frac 14}} & \left|\ffl{\frac s2}{\ue{z}}\right|^2\,dx 
	\\
	& \approx \varepsilon^{\frac 14}\int_{|z|>1} \left|e^{i\left[\frac{\xi_0}{\varepsilon}\left(x(t)+\varepsilon^{\frac 14}z\right)+\left(\frac{|\xi_0|}{\varepsilon}\right)^{2s}t\right]}\Big(|\xi_0|^s a_0 + c(s)\mathcal D^sa_0 + \varepsilon^s\ffl{\frac s2}{a_0} \Big)\right|^2\,dz
	\\
	&\leq \varepsilon^{\frac 14}\max\big\{|\xi_0|^s,c(s)\big\}\norm{a_0}{H^s(\RR)} + \varepsilon^{\frac 14+s}\norm{a_0}{H^s(\RR)} = \mathcal O(\varepsilon^{\frac 14}).
\end{align*}	
This concludes the proof.
\end{proof}

\section{Application to the analysis of control properties}\label{control_sec}

In this Section, we present an informal discussion on the application of the WKB construction obtained in this paper to the null-controllability of equation \eqref{main_eq}.  

It is by now well-known that geometric optics constructions for wave-like equations can be used for deriving controllability properties. These properties are usually formulated by means of an observability inequality, in which the total energy of the solutions is uniformly estimated by a partial measurement (typically, the portion of energy localized in a subset of the domain or of its boundary). In this framework, the existence of localized solutions gives sharp necessary conditions for the observability property to hold. Indeed, as it was remarked by Ralston in \cite{ralston1982gaussian}, in order to observe these solutions the observation set must intersect every ray. If this were not the case, one could construct a quasi solution along a ray that would not hit the observation set and which, being negligible outside an arbitrarily small neighborhood of the ray, could not be observed. This is the so-called Geometric Control Condition (GCC), which has been proved to be \textit{almost} sufficient by Bardos, Lebeau and Rauch in \cite{bardos1992sharp}, and necessary by Burq and G\'erard in \cite{burq1997condition}.

These principles applies also in the context of our fractional Schr\"odinger equation. Indeed, according to Theorem \ref{loc_thm} the quasi-solution $\ue{z}$ to \eqref{main_eq} are concentrated along the rays of geometric optics obtained by solving the Hamiltonian system \eqref{char_syst}. Therefore, they propagate with the group velocity of the plane wave solutions, which can be analyzed in terms of $s$ and of the frequency $\xi_0$. 

Recall that the rays for our equation have been introduced in Section \ref{prel_sec}, and in the physical domain are given by the curves  
\begin{align*}
	x(t) = x_0 \pm 2s|\xi_0|^{2s-1} t.
\end{align*}

Hence, the velocity of propagation of the quasi-solution $\ue{z}$ coincides with the one of the rays, and it is given by the quantity 
\begin{align*}
	v = |\dot{x}(t)| = 2s|\xi_0|^{2s-1}.
\end{align*}

Then, depending on the value of the parameter $s$, we can have there different behaviors (see also Figure \ref{velocity_fig}) 

\begin{itemize}
	\item[$\bullet$] For $s<1/2$, since $1-2s<0$, the velocity of propagation of the solutions decreases with the frequency.
	\item[$\bullet$] For $s=1/2$, the velocity of propagation of the solutions remains constant ($v=1$).
	\item[$\bullet$] For $s>1/2$, since $1-2s>0$, the velocity of propagation of the solutions increases with the frequency.
\end{itemize}

\begin{figure}[!h]
	\centering
	\subfloat[$s=0.1$]{
		\includegraphics[scale=0.3]{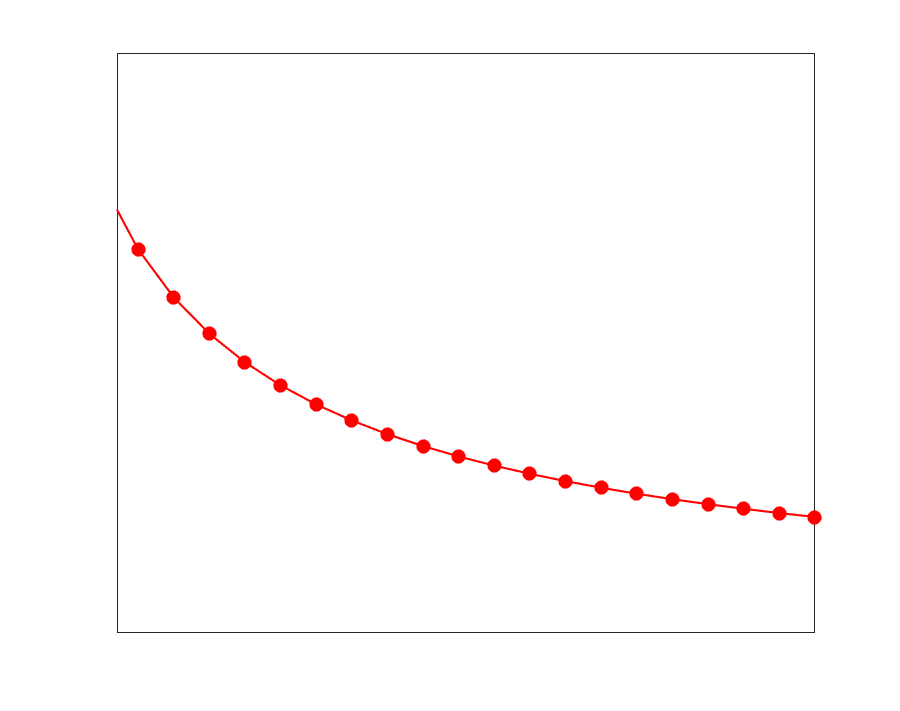}
		}\hspace{0.2cm}
		\subfloat[$s=0.5$]{
			\includegraphics[scale=0.3]{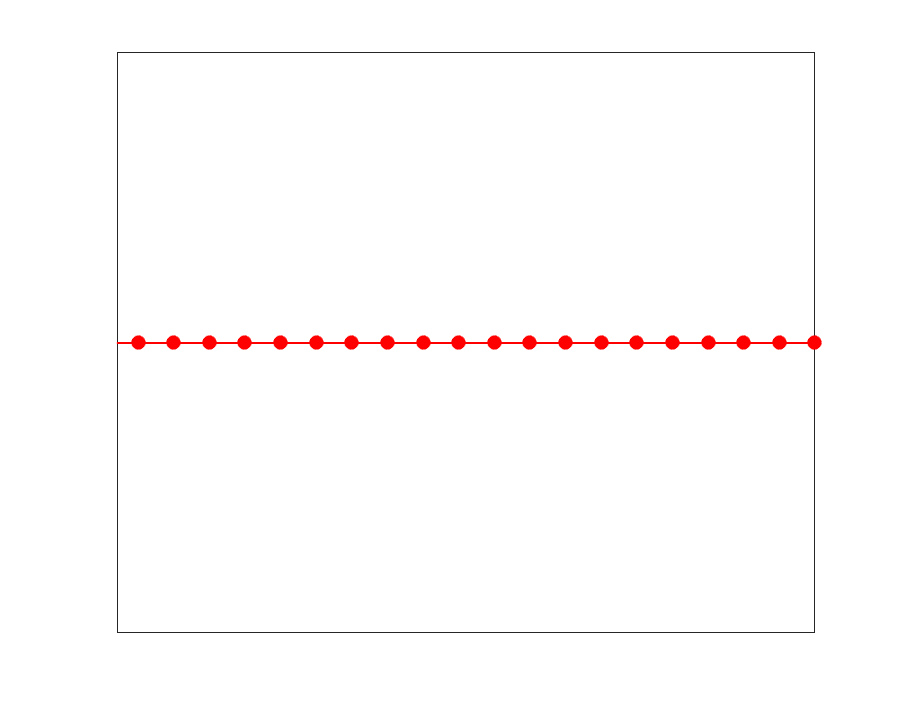}
		}
		\subfloat[$s=0.9$]{
			\includegraphics[scale=0.3]{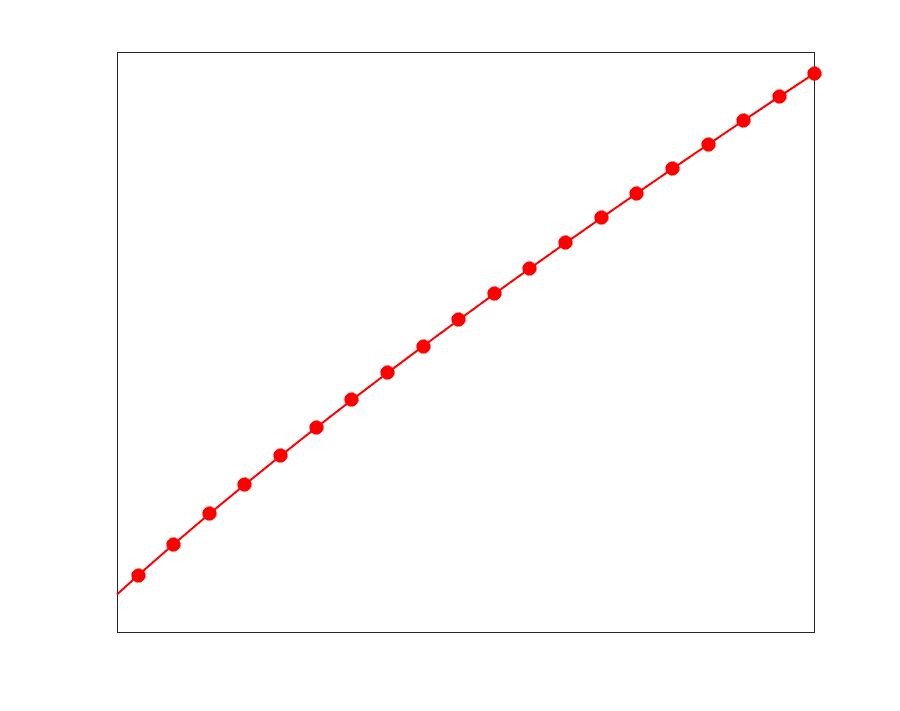}
		}\caption{Velocity of the rays as a function of $\xi_0$ for different values of $s\in(0,1)$}\label{velocity_fig}
\end{figure}

In view of that, the high frequency solutions are traveling faster and faster, for $s>1/2$, and slower and slower, for $s<1/2$. As a consequence, for $s>1/2$ the rays will travel sufficiently fast and are observable in any finite time $T>0$. For $s=1/2$, the velocity of propagation being constant, a minimum time $T_0$ is needed for the observation of all the rays. Finally, for $s<1/2$ the high frequency rays may not reach the control region, thus implying the failing of controllability properties. 

Lastly, we mention that this same behaviors have already been observed also in a multi-dimensional framework, in \cite{biccari2014internal}. Therefore, our informal discussion is not surprising, and it is somehow a confirmation of the results contained in the aforementioned paper.

\section{Numerical results}\label{numerical_sec}

We present here some simulations showing the propagation of solutions to the fractional Schr\"odinger equation \eqref{main_eq} corresponding to initial data in the form \eqref{in_dat}. 

For the numerical resolution of the equation, we employed a uniform mesh in the space variable and a FE discretization of the fractional Laplacian, obtained following the methodology presented in \cite{biccari2017controllability}. Moreover, we used a Crank-Nicholson scheme in time, which is known to be stable for the Schr\"odinger equation (see, e.g., \cite{askar1978explicit}). The initial data $u_0$ has been chosen as  
\begin{align*}
	u_0(x) = e^{-\frac \gamma2 (x-x_0)^2}e^{i\frac {\xi_0}{\varepsilon} x},
\end{align*} 
where the profile $u_{\textrm{in}}(x)$ is given by a Gaussian with standard deviation measured in terms of the parameter $\gamma$, which is related to the mesh size $h$. In particular we chose $\gamma = h^{-0.9}$. Finally, for the oscillations we considered frequencies $\xi_0=\pi^2/16$ and $\xi_0=2\pi^2$. We mention that this kind of approach is in many aspects analogous to what is done in \cite{biccari2018propagation,marica2015propagation}, where a similar analysis has been developed for the numerical finite difference solutions of one and two-dimensional waves, both with constant and variable coefficients.

In Figure \ref{plot_01}, we show the plots for $\xi_0=2\pi^2$ and different values of $s\in(0,1)$. The space domain has been chosen to be the interval $(-1,1)$, while we considered a time interval of $5$ seconds. 

\begin{figure}[h]
	\centering 
	\subfloat[Initial data]{
		\includegraphics[scale=0.15]{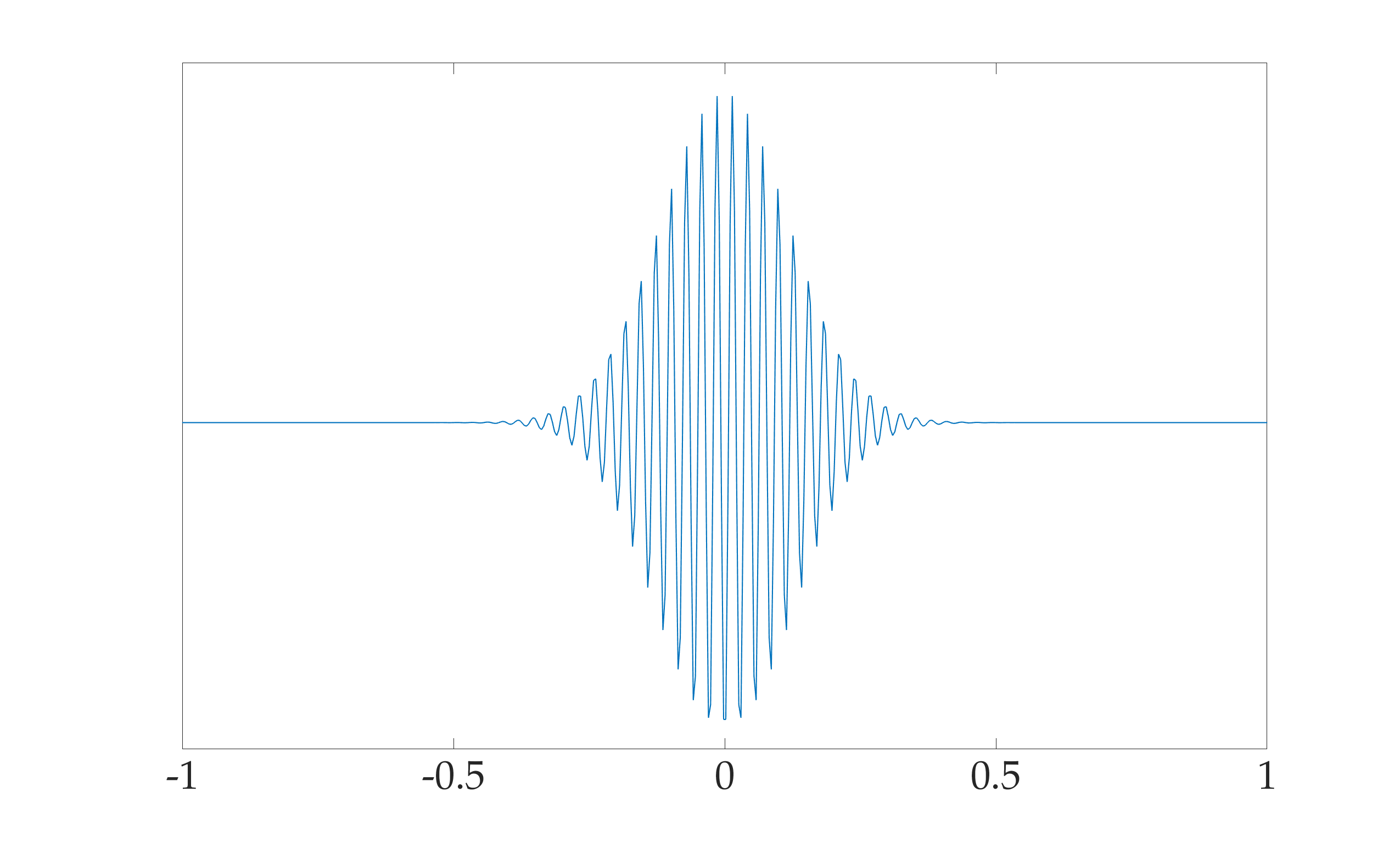}
	}\hspace{0.2cm}
	\subfloat[$s=0.1$]{
		\includegraphics[scale=0.15]{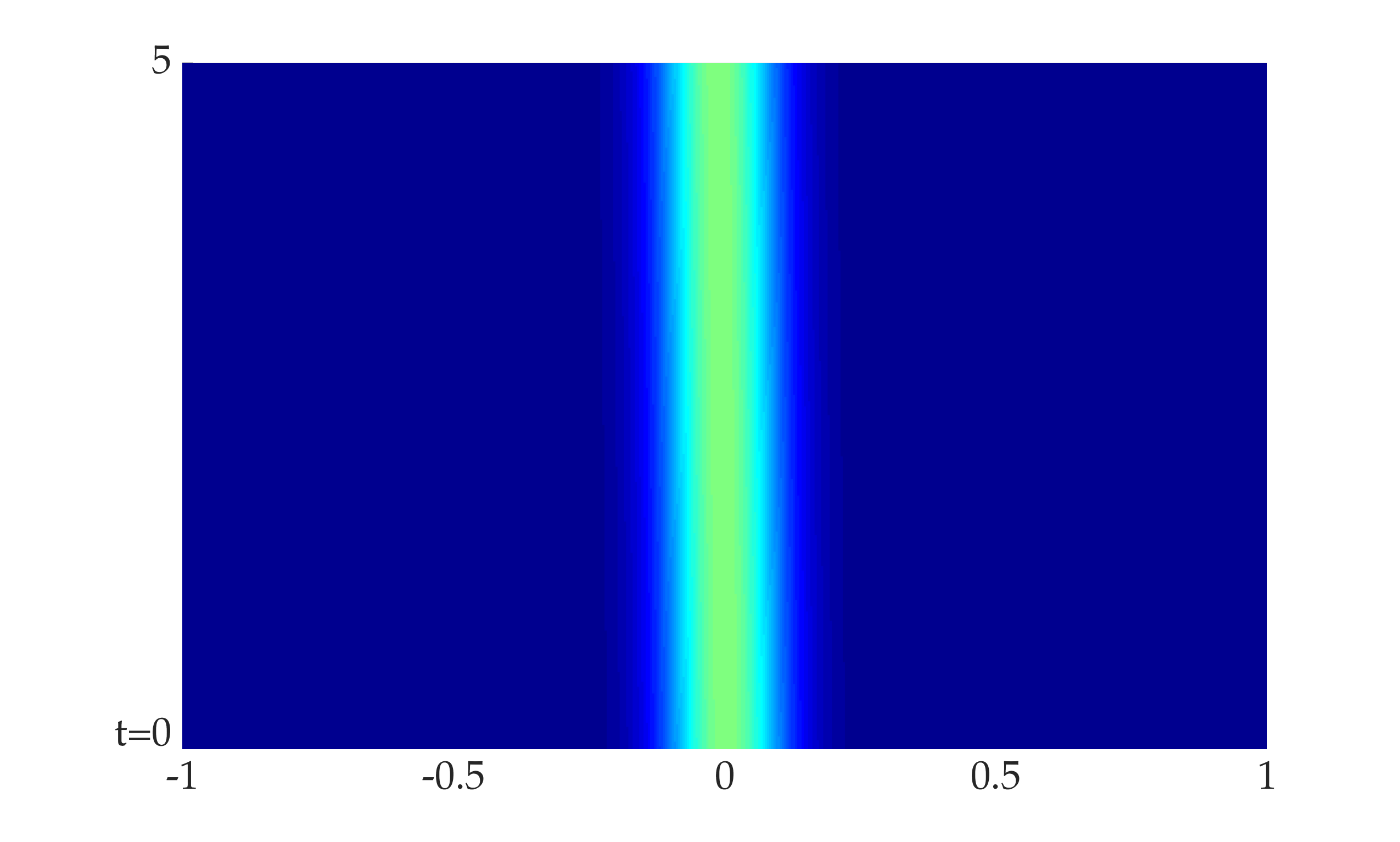}
	}\\
	\subfloat[$s=0.5$]{
		\includegraphics[scale=0.15]{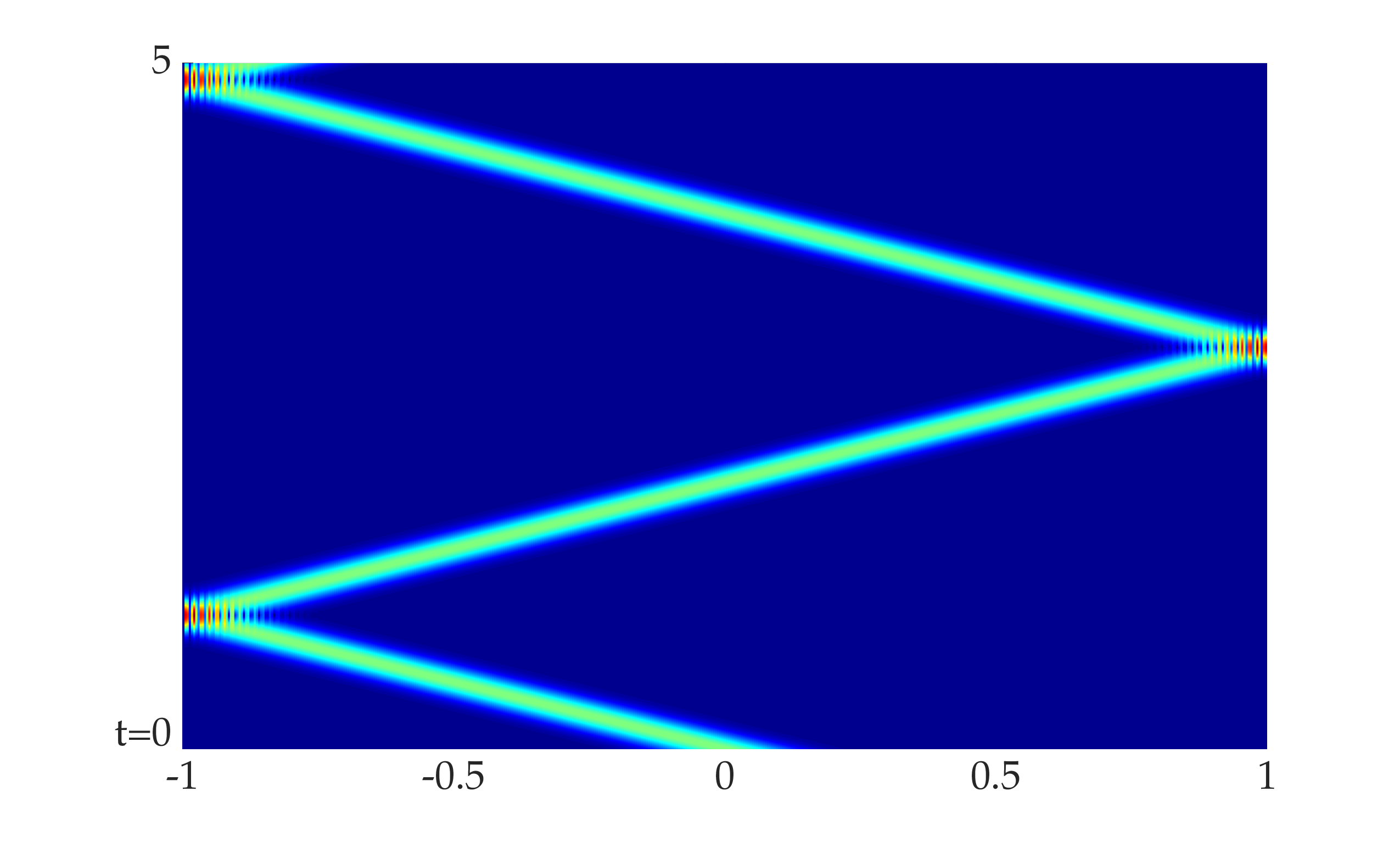}
	}\hspace{0.2cm}	
	\subfloat[$s=0.9$]{
		\includegraphics[width=6.5cm,height=3.9cm]{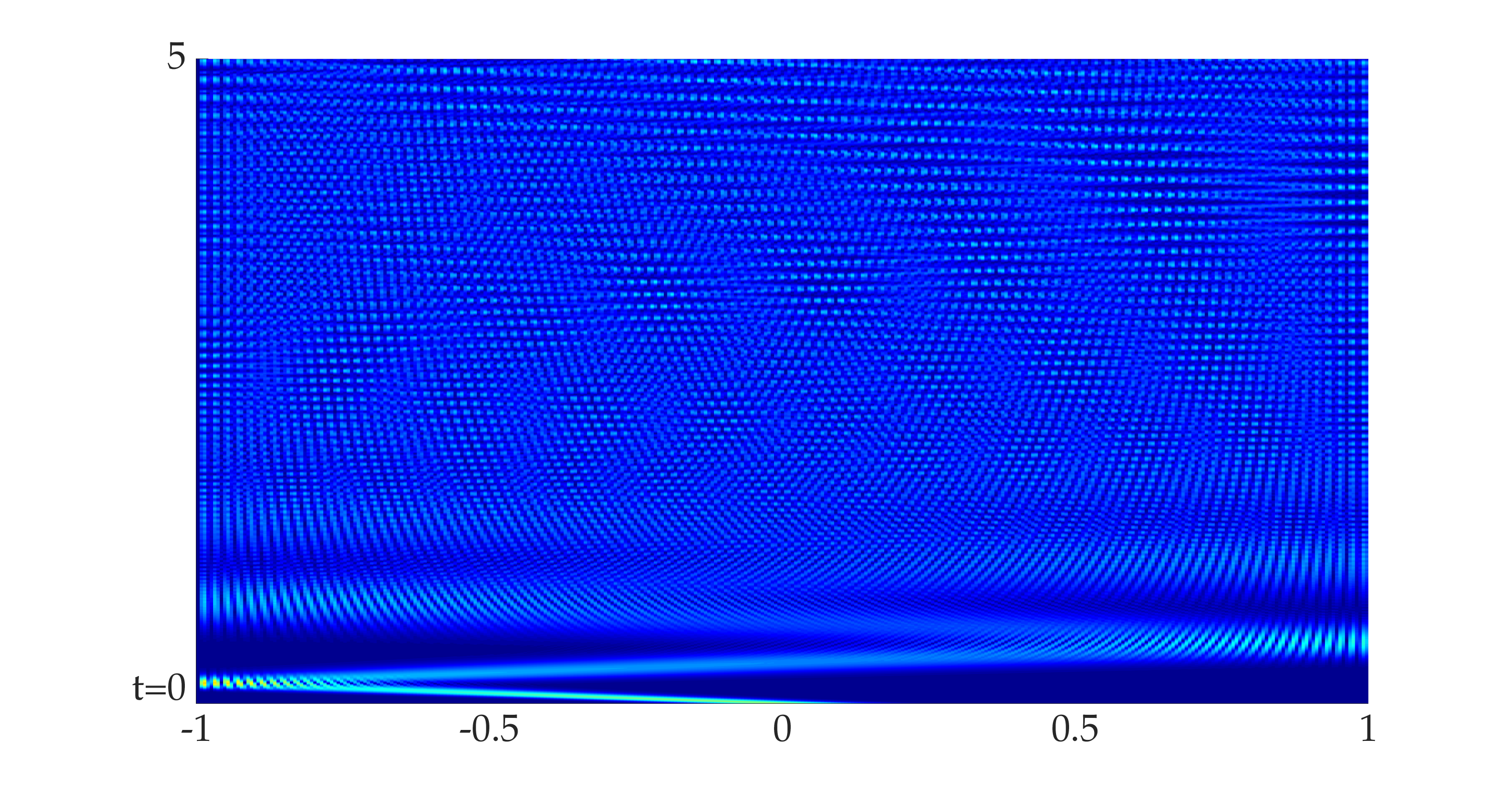}
	}\caption{Propagation of the solution for $\xi_0 = 2\pi^2$ and different values of $s$.}\label{plot_01}
\end{figure} 
It is seen there that, for small values of $s$, say $s=0.1$, the solution remains concentrated along rays which propagate only in the vertical direction. In other words, there is no propagation in space and, as we mentioned before, this implies that it will not be possible to control these solutions, no matter how one places the controls. For $s=0.5$, instead, the plots show that the solutions propagate along rays which reach the boundary of the space domain in finite time, and are reflected according to the laws of optics. This translates in the fact that, provided that the time is large enough, it will be possible to control these solutions, acting  with a control distributed in a neighborhood $\omega$ of the boundary. The case of high values of the power $s$ of the fractional Laplacian is the most puzzling one. For instance, for $s=0.9$ our simulations seem to show a lost of concentration of the solution along the ray, while our theoretical results would suggest that this concentration is preserved. On the other hand, we believe that what the simulations are showing is not necessarily in contradiction with the theory. We are going to address this issue in more details a later time. 

\begin{figure}[h]
	\centering 
	\subfloat[Initial data]{
		\includegraphics[scale=0.15]{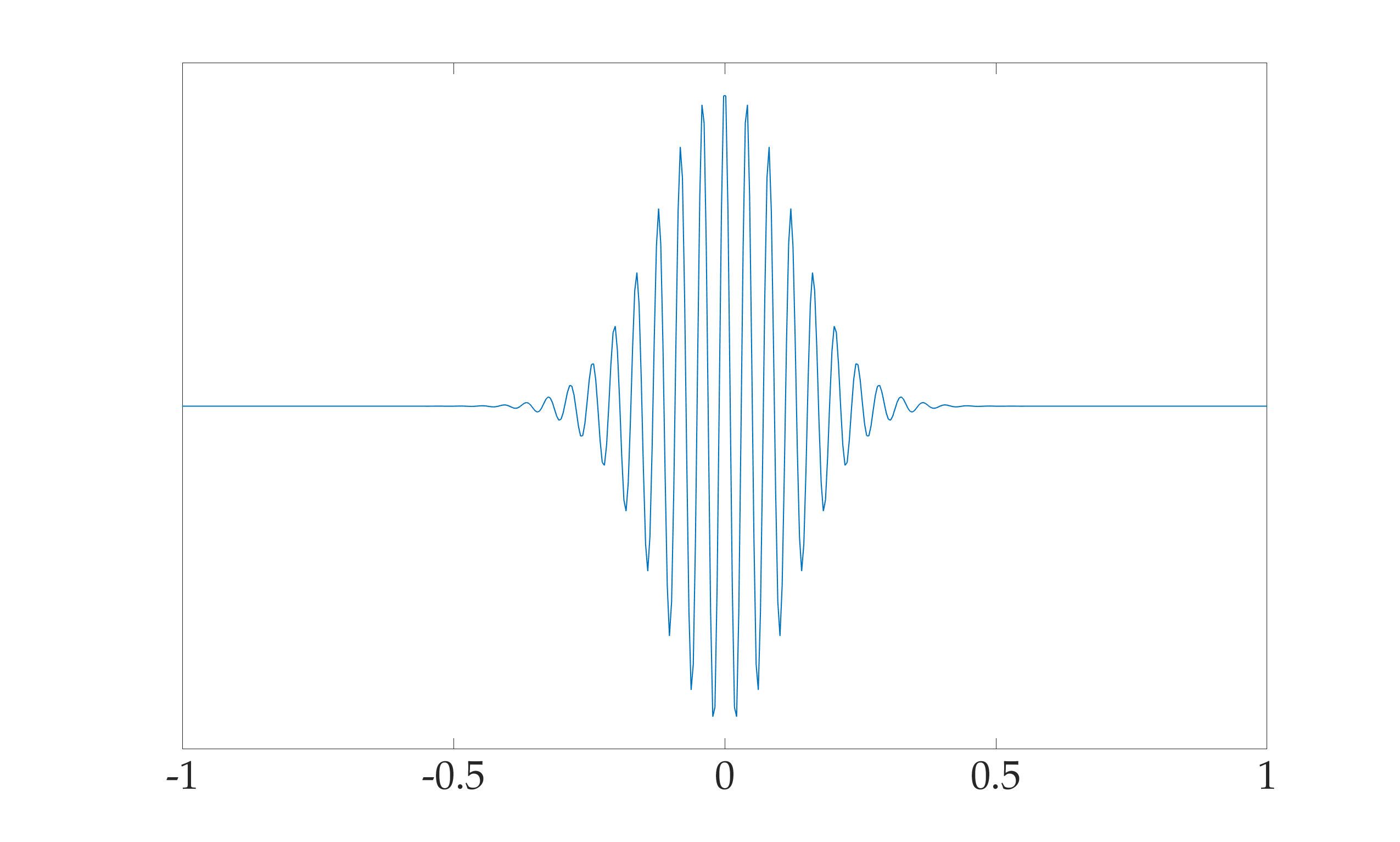}
	}\hspace{0.2cm}
	\subfloat[$s=0.1$]{
		\includegraphics[scale=0.15]{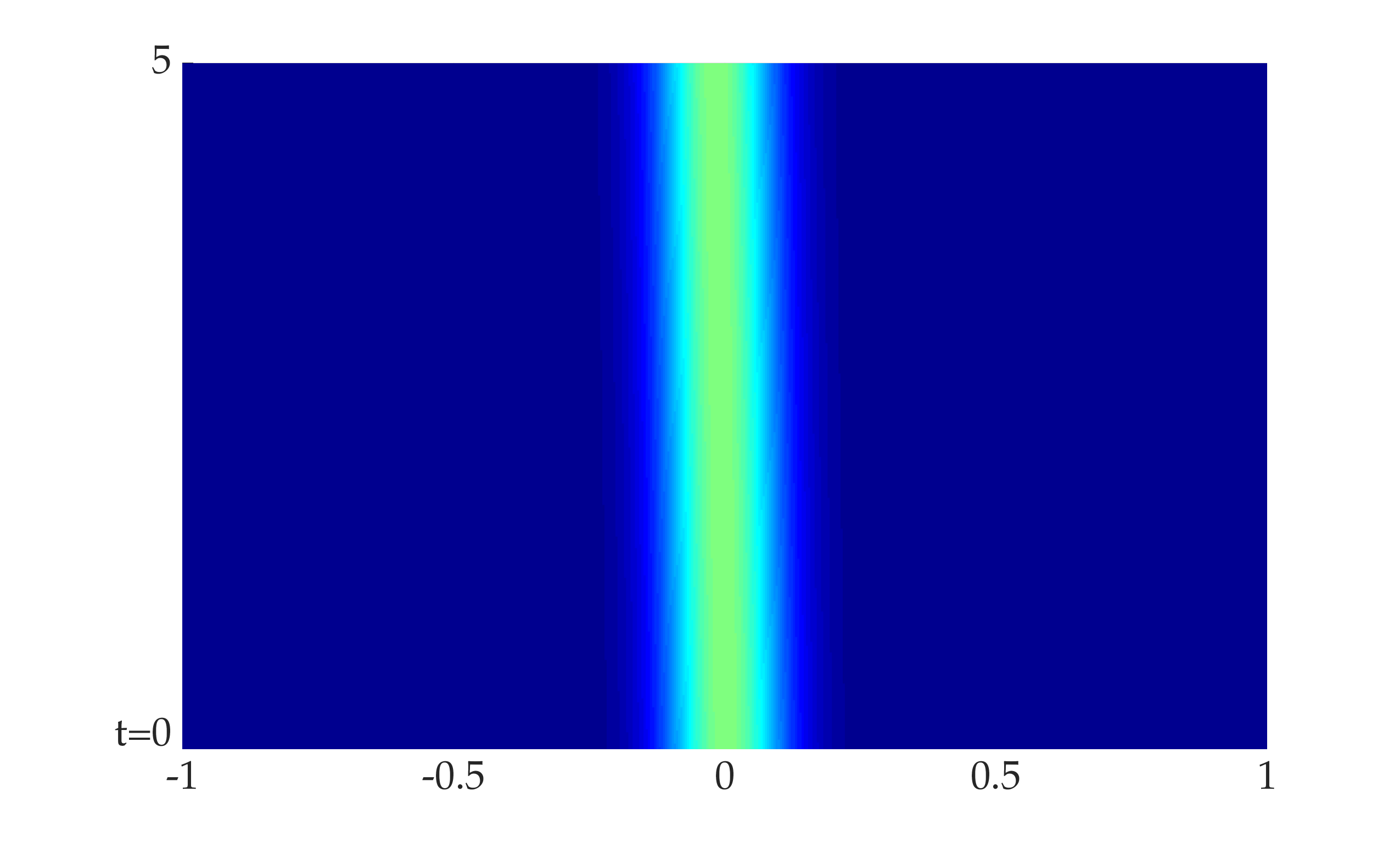}
	}\\
	\subfloat[$s=0.5$]{
		\includegraphics[scale=0.15]{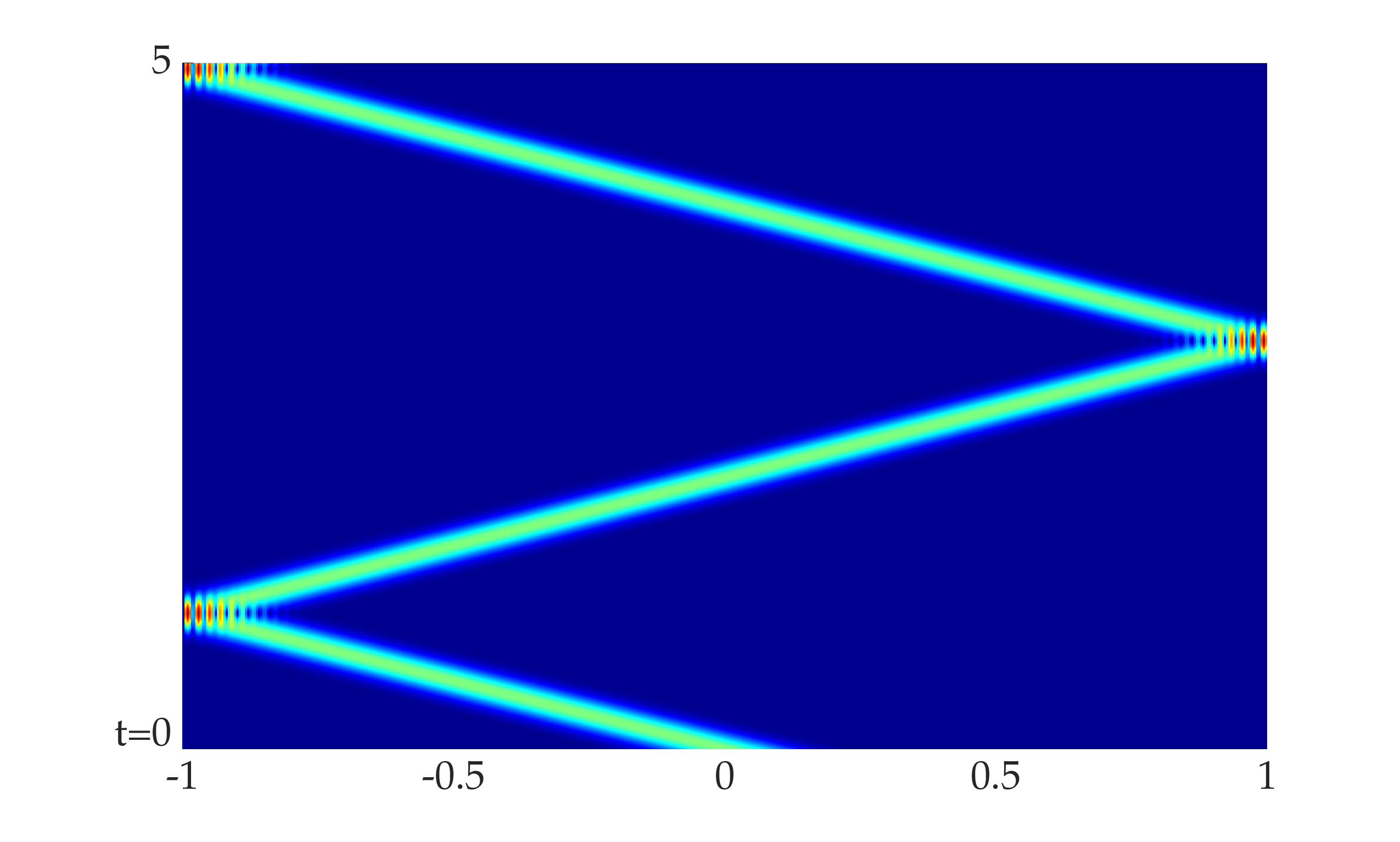}
	}\hspace{0.2cm}	
	\subfloat[$s=0.9$]{
		\includegraphics[width=6.5cm,height=3.9cm]{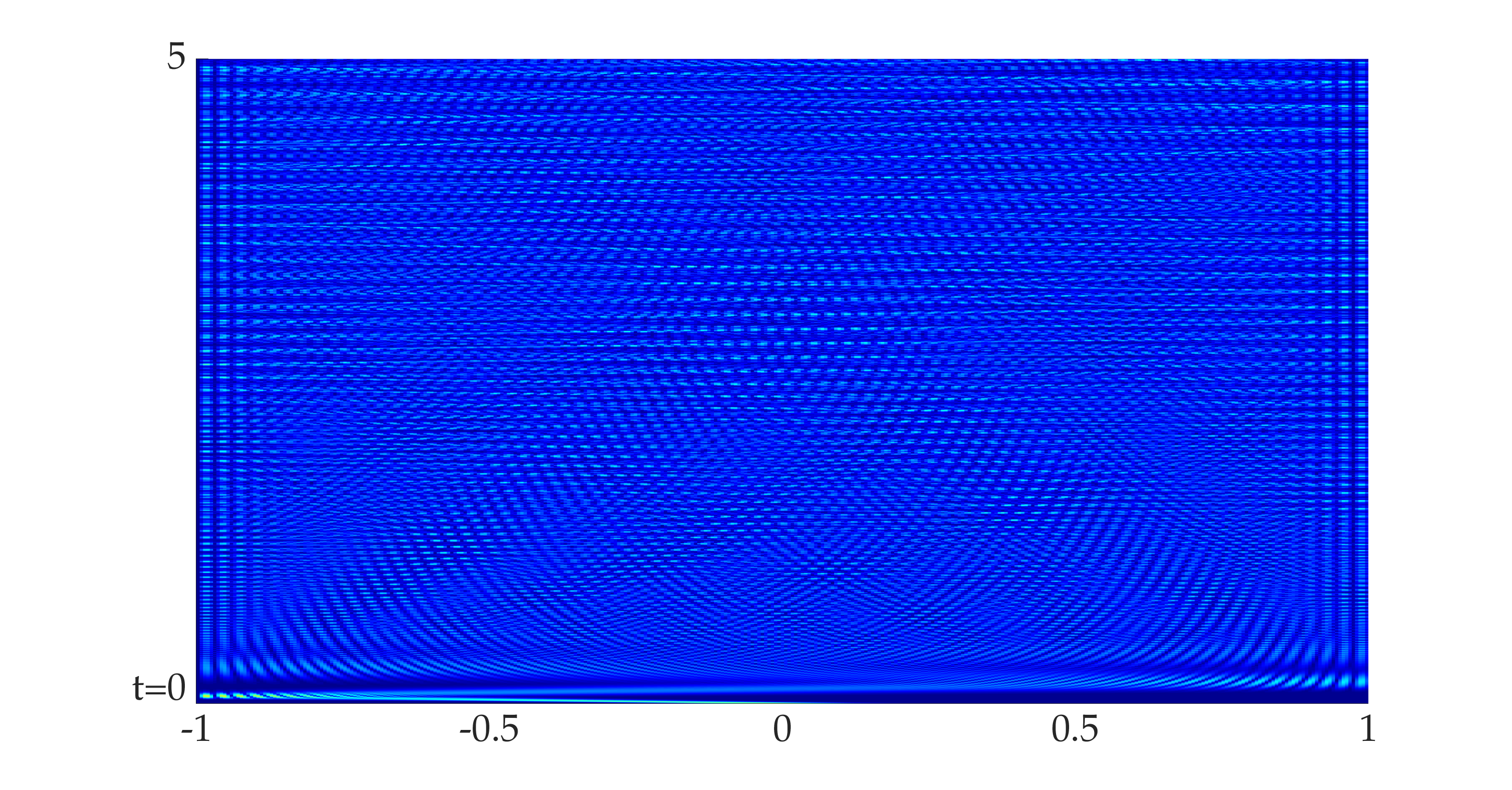}
	}\caption{Propagation of the solution for $\xi_0 = \pi^2/16$ and different values of $s$.}\label{plot_02}
\end{figure} 

In Figure \ref{plot_02}, the simulations have been run with an initial datum with frequency $\xi_0=\pi^2/16$. The plots obtained show a behavior which is totally analogous with what observed in Figure \ref{plot_01}:
\begin{itemize}
	\item[$\bullet$] For $s=0.1$, the solutions are once again concentrated along vertical rays, without propagation in time and, therefore, without possibility of being controlled.
	
	\item[$\bullet$] For $s=0.5$, we have propagation with constant velocity, and the ray reaches the boundary in finite time.
	
	\item[$\bullet$] For $s=0.9$ the chaotic comportment is still present.
\end{itemize}

Once again, the most surprising case is the last one, for $s>0.5$, in which the simulations seem to display dispersive features. Nevertheless, as we mentioned before, we retain that this does not contradict the results of Section \ref{loc_sec}. In our opinion, this strange phenomenon appearing in the plot can be explained with the accumulation of higher order terms in the asymptotic expansion of $\ue{z}$ which, combined with the small size of the space interval considered, enhance a chaotic behavior. This interpretation is supported by the fact that, as it is shown in Figure \ref{plot_03}, enlarging the space domain up to $(-6,6)$ seems to fix the problem and the localization of the solution along the rays appears once again. 

\begin{figure}[h]
	\centering 
	\subfloat[$\xi_0 = 2\pi^2$]{
		\includegraphics[scale=0.15]{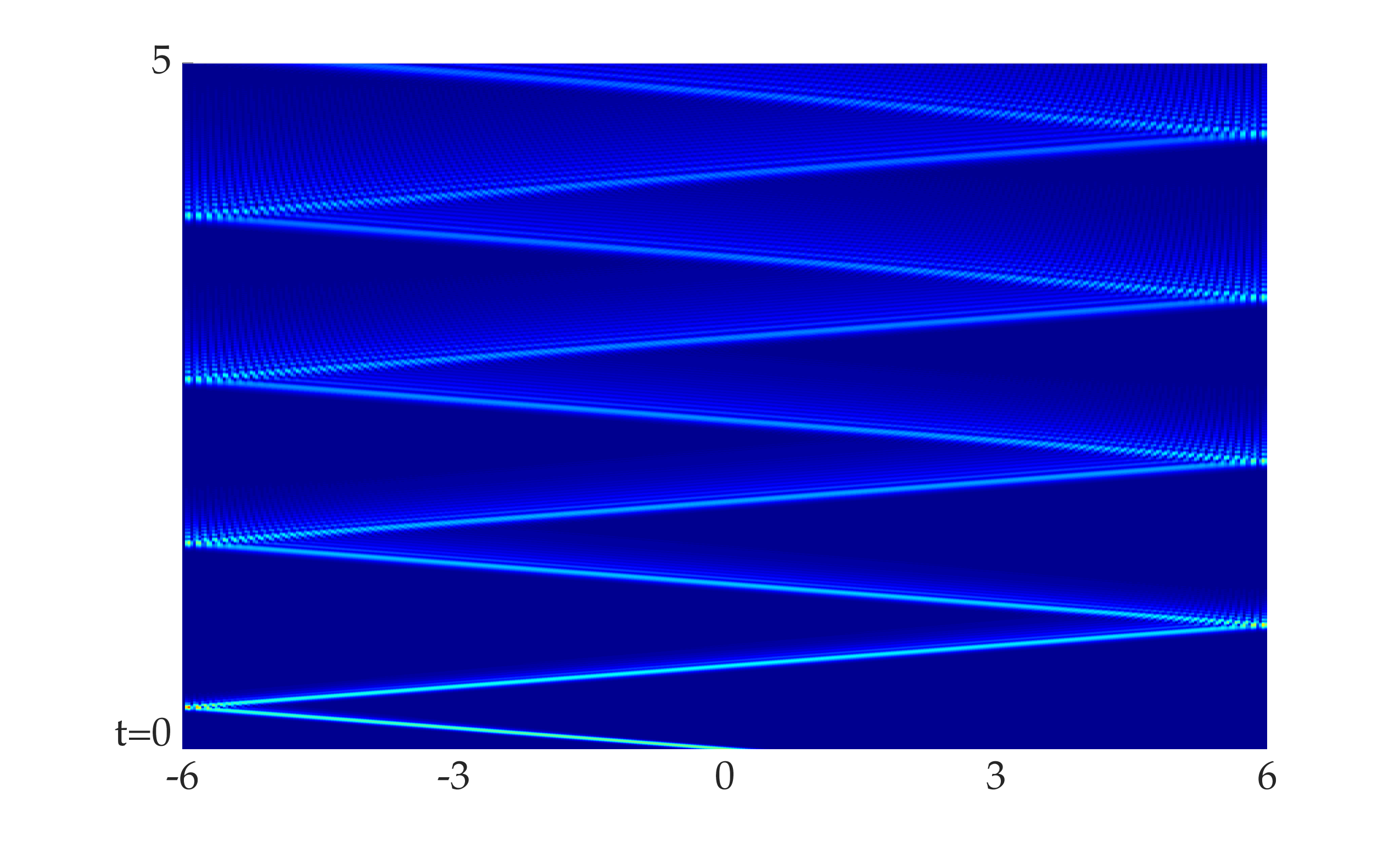}
	}\hspace{0.2cm}	
	\subfloat[$\xi_0 = \pi^2/16$]{
		\includegraphics[scale=0.15]{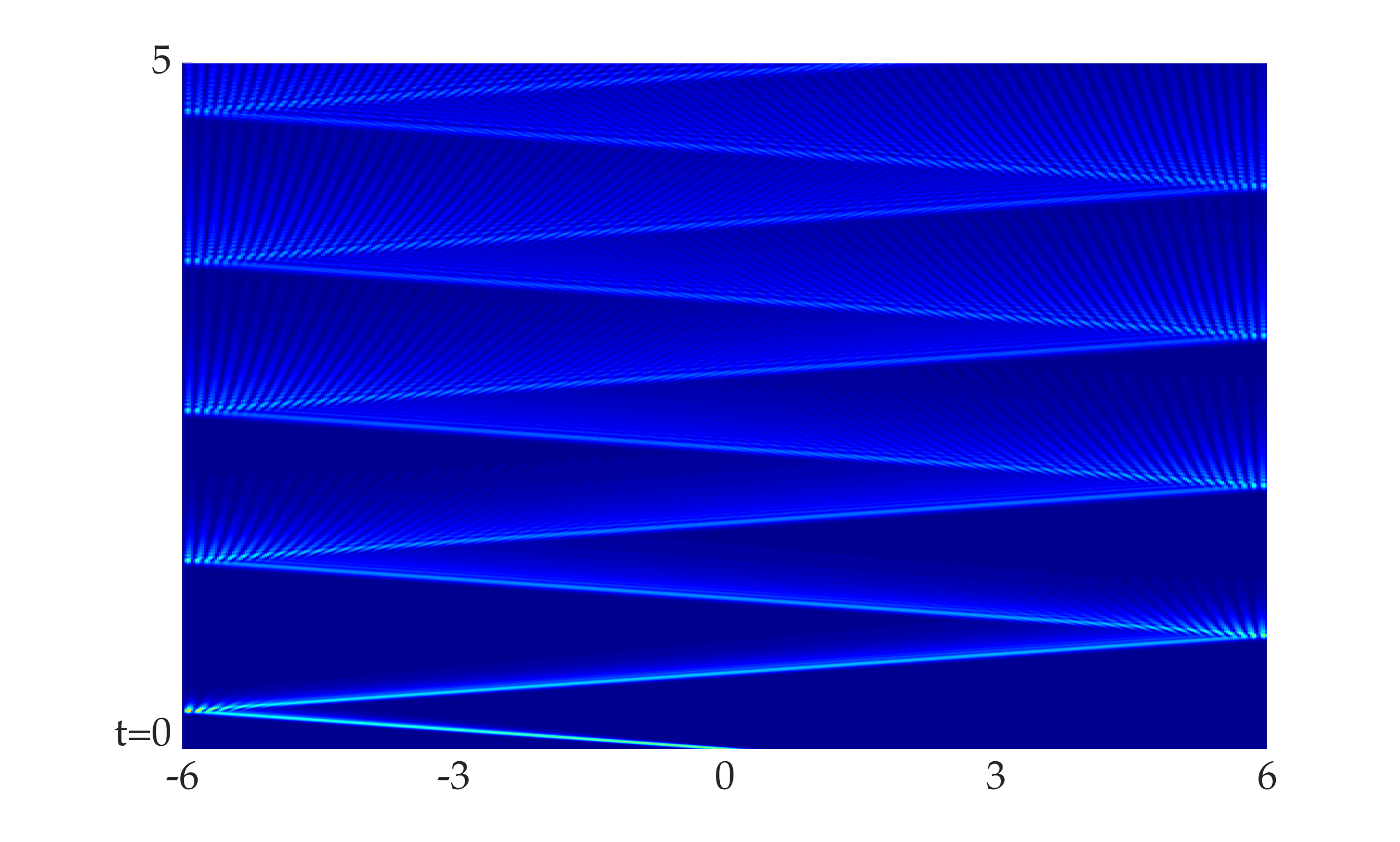}
	}\caption{Propagation of the solution for $s=0.9$ on the space interval $(-6,6)$.}\label{plot_03}
\end{figure} 

\section{Conclusions and final remarks}

In this article, we provided an explicit construction of a WKB ansatz for analyzing the propagation properties for a one-dimensional non-local Schr\"odigner equation involving the fractional Laplacian. In this way, we have been able to show that the solutions to the model that we considered are localized along the rays of geometric optics. In particular we proved that, the same as in the local case, the energy outside of a neighborhood of a ray is arbitrarily small. This fact can then be exploited for the study of controllability properties, in terms of the classical Geometric Control Condition. 

To the best of our knowledge, this kind of analysis had not been developed yet in a non-local setting and this constitutes the main contribution of our work. Nevertheless, we must mention that our approach is asymptotically correct though far from being optimal. 

Actually, it is commonly known that WKB is not the best technique for treating asymptotic analysis for wave-like PDEs. This because WKB solutions are usually subjected to \textit{caustic} formation, which occurs at those points where the characteristic flow $\Phi_t$ associated to the equation ceases to be a diffeomorphism. Depending on the model analyzed, this may happen in a finite time $T$, which might even be very small.

In general, this breakdown occurs when the density of rays becomes infinite. In that case, geometric optics incorrectly predicts that the amplitude of the solution is infinite. These problems are clearly not present in the exact solution of the model, but are
merely an artifact of the WKB ansatz. Caustics, therefore, indicate the appearance of new $\varepsilon$-oscillatory scales, which are not captured by the simple oscillatory ansatz.

The consideration of these difficulties led to the development of the theory of other techniques related to micro-local analysis and propagation of singularities, such as Fourier integral operators, Gaussian beams or a Wigner transformation approach. 

In particular, Gaussian beams is a high frequency asymptotic model which is closely related to geometrical optics. The quasi-solution are still assumed to be in the form
\begin{align*}
	\ue{z} = ae^{\frac i\varepsilon \phi},
\end{align*}  
but unlike the WKB approach, in this case the phase is complex-valued. The main advantage of this is that there is no breakdown at caustics, thus providing an approximation which remains valid globally in time.

On the other hand, in the case of equation \eqref{main_eq} considered in this paper, the choice of a complex-valued phase yields to technical difficulties at the moment of identifying the ansatz, which a the present stage we are not able to overcome. For this reason, we do not yet know how to develop a Guassian beam analysis for the solutions of our model.  

We would like to stress that, for the specific case of \eqref{main_eq}, we do not face the problem of caustics formation, since the Hamiltonian system \eqref{char_syst} admits a solution which is global in time. Nevertheless, this may not remain true when considering more general problems, for instance multi-dimensional or one with variable coefficients. For this reason, we believe that a different approach based on the aforementioned techniques should be developed in order to have a complete exhaustive understanding of the phenomena that we addressed in this paper.

\section{Acknowledgements}
This project has received funding from the European Research Council (ERC) under the European Union's Horizon 2020 research and innovation programme (grant agreement No. 694126-DyCon). The work of the first was partially supported by the Grants MTM2014-52347 and MTM2017-92996 of MINECO (Spain) and by the Grant FA9550-18-1-0242 of AFOSR.
	
The authors wish to acknowledge Enrique Zuazua (Autonomous University of Madrid, DeustoTech and University of Deusto - Bilbao, and University Pierre and Marie Curie - Paris) for having suggested the topic of this research and for interesting discussions. A special thank goes to Aurora Marica (Polytechnic University of Bucharest) for her help for obtaining the simulations presented in this work.
The second author (AA) thanks DeustoTech and the University of Deusto for hosting him in during summer of 2017, where part of this work was performed.

\bibliography{biblio}   
\end{document}